\documentclass[12pt]{amsart}
\usepackage{hyperref}
\usepackage{epsfig}
\newtheorem{theorem}{Theorem}[section]
\newtheorem{definition}[theorem]{Definition}
\newtheorem{proposition}[theorem]{Proposition}
\newtheorem{lemma}[theorem]{Lemma}
\newtheorem{corollary}[theorem]{Corollary}
\newtheorem{problem}[theorem]{Problem}
\newtheorem{conjecture}[theorem]{Conjecture}

\begin{document}

\title{Almost Hadamard matrices: the case of arbitrary exponents}

\author{Teodor Banica}
\address{T.B.: Department of Mathematics, Cergy-Pontoise University, 95000 Cergy-Pontoise, France. {\tt teodor.banica@u-cergy.fr}}

\author{Ion Nechita}
\address{I.N.: CNRS, Laboratoire de Physique Th\'eorique, IRSAMC, Universit\'e de Toulouse, UPS, 31062 Toulouse, France. {\tt nechita@irsamc.ups-tlse.fr}}

\subjclass[2000]{05B20 (15B10)}
\keywords{Hadamard matrix, Orthogonal group}

\begin{abstract}
In our previous work, we introduced the following relaxation of the Hadamard property: a square matrix $H\in M_N(\mathbb R)$ is called ``almost Hadamard'' if $U=H/\sqrt{N}$ is orthogonal, and locally maximizes the 1-norm on $O(N)$. We review our previous results, notably with the formulation of a new question, regarding the circulant and symmetric case. We discuss then an extension of the almost Hadamard matrix formalism, by making use of the $p$-norm on $O(N)$, with $p\in[1,\infty]-\{2\}$, with a number of theoretical results on the subject, and the formulation of some open problems.
\end{abstract}

\maketitle

\tableofcontents

\section*{Introduction}

An Hadamard matrix is a square matrix $H\in M_N(\pm 1)$ having its rows pairwise orthogonal. The Hadamard conjecture (HC), which is over a century old, states that such matrices exist, at any $N\in 4\mathbb N$. See \cite{aga}, \cite{hor}, \cite{kta}, \cite{lgo}. The circulant Hadamard conjecture (CHC), which is half a century old \cite{rys}, states that a circulant Hadamard matrix can exist only at $N=4$. More precisely, only the following matrix $K_4$ and its various ``conjugates'' can be at the same time circulant and Hadamard, regardless of the size $N\in\mathbb N$:
$$K_4=\begin{pmatrix}-1&1&1&1\\ 1&-1&1&1\\ 1&1&-1&1\\ 1&1&1&-1\end{pmatrix}$$

An interesting generalization of the Hadamard matrices are the complex Hadamard matrices, namely the matrices $H\in M_N(\mathbb T)$, where $\mathbb T$ is the unit circle, having their rows pairwise orthogonal. These matrices appear in several contexts, see \cite{ha1}, \cite{jon}, \cite{jsu}, \cite{pop}, \cite{tz1}, \cite{tao}, \cite{wer}. The main example is the rescaled Fourier matrix ($w=e^{2\pi i/N}$):
$$F_N=\begin{pmatrix}
1&1&1&\ldots&1\\
1&w&w^2&\ldots&w^{N-1}\\
\ldots&\ldots&\ldots&\ldots&\ldots\\
1&w^{N-1}&w^{2(N-1)}&\ldots&w^{(N-1)^2}
\end{pmatrix}$$

This example prevents the existence of a complex analogue of the HC. However, when trying to build complex Hadamard matrices with roots of unity of a given order, a subtle generalization of the HC problematics appears \cite{but}, \cite{lle}, \cite{lau}. In relation now with the CHC, there has been some interesting work here on the circulant case \cite{bjo}, \cite{bfr}, \cite{ha2}. Also, much work has gone into various geometric aspects, see \cite{ban}, \cite{bbe}, \cite{krs}, \cite{szo}, \cite{tz2}.

Yet another generalization comes from \cite{bcs}, \cite{bnz}. The original observation from \cite{bcs} is that for an orthogonal matrix $U\in O(N)$ we have $||U||_1\leq N\sqrt{N}$, with equality if and only if $H=\sqrt{N}U$ is Hadamard. This follows indeed from the Cauchy-Schwarz inequality:
$$||U||_1=\sum_{ij}|U_{ij}|\leq N\left(\sum_{ij}U_{ij}^2\right)^{1/2}=N\sqrt{N}$$

This simple fact suggests that a natural and useful generalization of the Hadamard matrices are the matrices of type $H=\sqrt{N}U$, with $U\in O(N)$ being a maximizer of the 1-norm. However, since such matrices are quite difficult to approach, most efficient is to study first the matrices of type $H=\sqrt{N}U$, with $U\in O(N)$ being just a local maximizer of the 1-norm. Such matrices are called ``almost Hadamard''. See \cite{bnz}.

One key feature of the almost Hadamard matrices is that at the level of examples we have a number of infinite series, uniformly depending on $N\in\mathbb N$. The basic example is:
$$K_N=\frac{1}{\sqrt{N}}
\begin{pmatrix}
2-N&2&\ldots&2\\
2&2-N&\ldots&2\\
\ldots\\
2&2&\ldots&2-N
\end{pmatrix}$$

Observe that $K_N$ is circulant, and that $K_4$ is Hadamard. Thus we are quickly led into the circulant Hadamard matrix problematics, and we have the following questions:

\medskip

\noindent {\bf Problem.} {\em What are the circulant Hadamard matrices? The circulant complex Hadamard matrices? The circulant almost Hadamard matrices?}

\medskip

More precisely, the CHC states that there are exactly 8 circulant Hadamard matrices, namely $K_4$ and its conjugates. Regarding the second question, Haagerup has shown in \cite{ha2} that for $N=p$ prime, the number of circulant complex Hadamard matrices, counted with certain multiplicities, is exactly $\binom{2p-2}{p-1}$, and the problem is to see what happens when $N$ is not prime. As for the third question, this appears from our previous work \cite{bnz}.

\begin{figure}[htbp]
\includegraphics[scale=0.7]{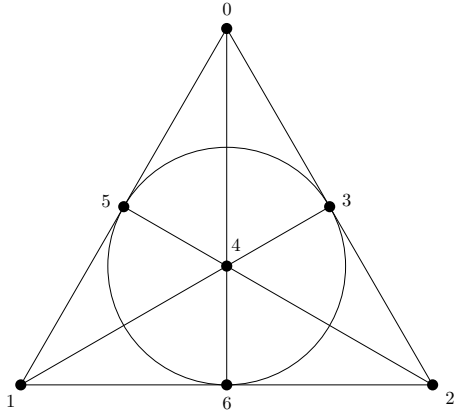}
\caption{The Fano plane}   
\end{figure}

Regarding this latter question, it was shown in \cite{bnz} that we have a number of interesting examples coming from block designs \cite{cdi}, \cite{sti}. The simplest one, coming from the adjacency matrix of the Fano plane (see the Figure), is as follows, with $x=2-4\sqrt{2}$, $y=2+3\sqrt{2}$:
$$I_7=\frac{1}{2\sqrt{7}}\begin{pmatrix}
x&x&y&y&y&x&y\\
y&x&x&y&y&y&x\\
x&y&x&x&y&y&y\\
y&x&y&x&x&y&y\\
y&y&x&y&x&x&y\\
y&y&y&x&y&x&x\\
x&y&y&y&x&y&x
\end{pmatrix}$$

Now back to the above 3 questions, the point is that, from the point of view of Fourier analysis, these are all related. Indeed, with $F=F_N/\sqrt{N}$, the circulant unitary matrices are precisely those of the form $U=FQF^*$ with $Q$ belonging to the torus $\mathbb T^N$ formed by the diagonal matrices over $\mathbb T$. So, in view of the above-mentioned remark about the 1-norm, all the above questions concern the understanding of the following potential:
$$\Phi:\mathbb T^N\to[0,\infty)$$
$$\ \ \ \ \ \ \ \ Q\to||FQF^*||_1$$

With this approach, the first thought goes to the computation of the moments of $\Phi$. Indeed, the global maximum, or more specialized quantities such as the exact number of maxima, can be recovered via variations of the following well-known formula:
$$\max(\Phi)=\lim_{k\to\infty}\left(\int_{\mathbb T^N}\Phi^k\right)^{1/k}$$

Of course, in respect to the above problems, one has to restrict sometimes attention to the torus $\mathbb T^n\subset\mathbb T^N$, with $n=\lfloor N/2 \rfloor + 1$, coming from the orthogonal matrices.

The origins of this approach go back to \cite{bcs}, where the potential $\Phi(U)=||U||_1$ was investigated over the group $O(N)$, in connection with the HC. Of course, the computation of moments over $O(N)$ is a quite complicated question \cite{bsc}, \cite{csn}. In the circulant case, however, the parameter space being just $\mathbb T^N$, the integration problem is much simpler. But it still remains very complicated, and we have no concrete results here so far.

So, instead of trying to understand the potential $\Phi:\mathbb T^N\to[0,\infty)$ directly via its moments and analysis, we should rather try to first develop a few geometric techniques. The point is that $\Phi$ has a number of symmetries, and when investigating these symmetries, the lattice $\{\pm 1\}^N\subset\mathbb T^N$ coming from the self-adjoint matrices seems to play a key role.

More precisely, we will study here the circulant and symmetric orthogonal matrices, which correspond via Fourier transform to the sequences $\alpha\in\{\pm 1\}^N$ satisfying $\alpha_i=\alpha_{-i}$. Our result here, motivated by the ``almost Hadamard'' problematics, is as follows:

\medskip

\noindent {\bf Proposition.} {\em Any circulant and symmetric matrix $U\in O(N)^*$ is a critical point of all $p$-norms on $O(N)$. The local maximizers of the $1$-norm can be counted up to $N=30$.}

\medskip

In this statement, $O(N)^*\subset O(N)$ is the set of orthogonal matrices having nonzero entries. For more comments on this result, we refer to the body of the paper.

Back to the general case now, one observation from \cite{bcs} is that one can replace the 1-norm by the $p$-norm, for any $p\neq 2$. Indeed, at $p<2$ the H\"older inequality gives:
$$||U||_p=\left(\sum_{ij}|U_{ij}|^p\right)^{1/p}\leq N^{2/p-1}\left(\sum_{ij}U_{ij}^2\right)^{1/2}=N^{2/p-1/2}$$

Thus for $U\in O(N)$ we have $||U||_p\leq N^{2/p-1/2}$, with equality if and only if the rescaled matrix $H=\sqrt{N}U$ is Hadamard. In the $p>2$ case a similar result holds, with the H\"older inequality applying in the reverse sense, and giving the estimate $||U||_p\geq N^{2/p-1/2}$.

So, we are led to the following notion, generalizing those in \cite{bcs}, \cite{bnz}:

\medskip

\noindent {\bf Definition.} {\em A square matrix $M\in M_N(\mathbb R)$ is called ``$p$-almost Hadamard'' if the rescaled matrix $U=H/\sqrt{N}$ is orthogonal, and is a local extremum of the $p$-norm on $O(N)$.}

\medskip

Here by ``local extremum'' we mean a local maximum at $p<2$, and a local minimum at $p>2$.  Also, we will call $H$ ``optimal'' if $U=H\sqrt{N}$ is a global maximum/minimum.  

One interest in this generalization comes from the exponent $p=4$, believed to be of interest in relation with quantum physics questions. Indeed, for $U\in O(N)$ we have:
$$d_{HS}\left((U_{ij}^2),J_N\right)^2
=\sum_{ij}\left(U_{ij}^2-\frac{1}{N}\right)^2 
=||U||_4^4-1$$

This computation shows that an orthostochastic matrix $A$ is ``almost flat'', in the sense that it minimizes the Hilbert-Schmidt distance to the flat matrix $J_N$, if and only if $A_{ij}=U_{ij}^2$, with $U$ being a global minimizer of the $4$-norm on $O(N)$. See \cite{be+}, \cite{dit}.

We will prove here that the $p$-almost Hadamard matrices can be detected by using linear algebra. We conjecture that the basic matrix $K_N$ always has this property.

The paper is organized as follows: in 1-2 we review the material in \cite{bnz}, and we discuss some new questions in the circulant case, and in 3-4 we perform some differential geometry computations, and we present a list of questions, that we don't know how to solve.

\medskip

\noindent {\bf Acknowledgements.} The present article is part of a series started in \cite{bcs}, in collaboration with Beno\^ it Collins and Jean-Marc Schlenker, and recently continued in \cite{bnz}, in collaboration with Karol \.Zyczkowski. It is our pleasure to thank all three for endless discussions and encouragements, and particularly Karol \.Zyczkowski for numerous recent discussions on the subject. The work of I.N. was supported by the ANR grant BS01 008 01.

\section{Almost Hadamard matrices}

We consider in this paper various square matrices $M\in M_N(\mathbb C)$. The indices of our matrices will vary in $\{0,1,\ldots,N-1\}$, and will be sometimes taken modulo $N$.

As explained in the introduction, a direct application of the Cauchy-Schwarz inequality shows that for $O(N)$ we have $||U||_1\leq N\sqrt{N}$, with equality if and only $H=\sqrt{N}U$ is Hadamard. In \cite{bnz} we have introduced the notion of almost Hadamard matrix:

\begin{definition}
An ``almost Hadamard'' matrix is a square matrix $H\in M_N(\mathbb R)$ such that $U=H/\sqrt{N}$ is orthogonal, and is a local maximum of the $1$-norm on $O(N)$. Equivalently, $U_{ij}\neq 0$, and the matrix $SU^t$, with $S_{ij}={\rm sgn}(U_{ij})$, must be positive definite.
\end{definition}

In this statement the equivalence between the two conditions comes from a number of differential geometry computations, for which we refer to \cite{bcs}, or \cite{bnz}.

As a first remark, any Hadamard matrix is almost Hadamard. In particular, given a number $N\in\{2\}\cup 4\mathbb N$ where the Hadamard Conjecture holds, the Hadamard matrices of order $N$ are precisely the almost Hadamard matrices $H\in M_N(\mathbb R)$ which are ``optimal'', in the sense that $U=H/\sqrt{N}$ is a global maximum of the 1-norm on $O(N)$.

The above definition provides a useful, flexible generalization of the quite rigid class formed by the Hadamard matrices. For instance at any $N\geq 3$ we have a number of concrete examples, which can be used for various purposes. The basic example is:
$$K_N=\frac{1}{\sqrt{N}}
\begin{pmatrix}
2-N&2&\ldots&2\\
2&2-N&\ldots&2\\
\ldots\\
2&2&\ldots&2-N
\end{pmatrix}$$

This matrix has several remarkable properties, for instance it is at the same time circulant and symmetric. Also, it has at most 2 entries, which are both nonzero.

So, let us look now more in detail at the matrices having similar properties. We recall from \cite{bnz} that an $(a,b,c)$ pattern is a matrix $M\in M_N(x,y)$, with $N=a+2b+c$, such that any two rows of our matrix look as follows, up to a permutation of the columns:
$$\begin{matrix}
x\ldots x&x\ldots x&y\ldots y&y\ldots y\\
\underbrace{x\ldots x}_a&\underbrace{y\ldots y}_b&\underbrace{x\ldots x}_b&\underbrace{y\ldots y}_c
\end{matrix}$$

Observe that the above matrix $K_N$ comes from a $(0,1,N-2)$ pattern. There are many other examples, the main result here being that the adjacency matrix of any $(N,a+b,a)$ symmetric balanced incomplete block design is an $(a,b,c)$ pattern. See \cite{bnz}.

The following result was proved in \cite{bnz}: 

\begin{proposition}
Let $U=U(x,y)$ be orthogonal, coming from an $(a,b,c)$ pattern.
\begin{enumerate}
\item $U$ is a critical point of the $1$-norm on $O(N)$.

\item $H=\sqrt{N}U$ is almost Hadamard iff $(N(a-b)+2b)|x|+(N(c-b)+2b)|y|\geq 0$.
\end{enumerate}
\end{proposition}

\begin{proof}
Since any row of $U$ consists of $a+b$ copies of $x$ and $b+c$ copies of $y$, we get:
$$(SU^t)_{ij}
=\begin{cases}
(a+b)|x|+(b+c)|y|&(i=j)\\
(a-b)|x|+(c-b)|y|&(i\neq j)
\end{cases}$$

Thus $SU^t$ is symmetric, and by \cite{bcs} our matrix $U$ is a critical point of the 1-norm. Regarding now the second assertion, observe first that we can write $SU^t$ as follows:
\begin{eqnarray*}
SU^t
&=&2b(|x|+|y|)1_N+((a-b)|x|+(c-b)|y|)NJ_N\\
&=&2b(|x|+|y|)(1_N-J_N)+((N(a-b)+2b)|x|+(N(c-b)+2b)|y|))J_N
\end{eqnarray*}

Now since  $1_N-J_N,J_N$ are orthogonal projections, we have $SU^t>0$ if and only if the coefficients of these matrices are both positive, and this gives the result.
\end{proof}

Let us go back now to our observation that $K_N$ is at the same time circulant and symmetric, and look in detail at the matrices having these two properties. We fix $N\in\mathbb N$, and we denote by $F=F_N \in U(N)$ the Fourier matrix, $F=(w^{ij})/\sqrt{N}$ with $w=e^{2\pi i/N}$. Also, given a vector $q\in\mathbb C^N$, we associate to it the diagonal matrix $Q=\mathrm{diag}(q_0,\ldots,q_{N-1})$.

\begin{lemma}
For a matrix $H\in M_N(\mathbb C)$, the following are equivalent:
\begin{enumerate}
\item $H$ is circulant, i.e. $H_{ij}=\gamma_{j-i}$, for a certain vector $\gamma\in\mathbb C^N$.

\item $H$ is Fourier-diagonal, i.e. $H=FDF^*$, with $D\in M_N(\mathbb C)$ diagonal.
\end{enumerate}
In addition, if so is the case, then with $D=\sqrt{N}Q$ we have $\gamma=Fq$.
\end{lemma}

\begin{proof}
This result is well-known, and the proof goes as follows:

(1)$\implies$(2) The matrix $D=F^*HF$ is indeed diagonal, given by:
$$D_{ij}=\frac{1}{N}\sum_{kl}w^{jl-ik}\gamma_{l-k}=\delta_{ij}\sum_rw^{jr}\gamma_r$$ 

(2)$\implies$(1) The matrix $H=FDF^*$ is indeed circulant, given by:
$$H_{ij}=\sum_kF_{ik}D_{kk}\bar{F}_{jk}=\frac{1}{N}\sum_kw^{(i-j)k}D_{kk}$$

Finally, the last assertion is clear from the above formula of $H_{ij}$.
\end{proof}

The following result is as well from \cite{bnz}:

\begin{proposition}
A circulant matrix $H\in M_N(\mathbb R^*)$, written $H_{ij}=\gamma_{j-i}$, is almost Hadamard if and only if the following conditions are satisfied:
\begin{enumerate}
\item The vector $q=F^*\gamma$ satisfies $q\in\mathbb T^N$.

\item With $\varepsilon={\rm sgn}(\gamma)$, $\rho_i=\sum_r\varepsilon_r\gamma_{i+r}$ and $\nu=F^*\rho$, we have $\nu>0$.
\end{enumerate}
\end{proposition}

\begin{proof}
By Lemma 1.3 the orthogonality of $U$ is equivalent to the condition (1). Regarding now the condition $SU^t>0$, this is equivalent to $S^tU>0$. But:
$$(S^tH)_{ij}=\sum_kS_{ki}H_{kj}=\sum_k\varepsilon_{i-k}\gamma_{j-k}=\sum_r\varepsilon_r\gamma_{j-i+r}=\rho_{j-i}$$

Thus $S^tU$ is circulant, with $\rho/\sqrt{N}$ as first row. From Lemma 1.3 again we get $S^tU=FLF^*$ with $L=diag(\nu)$ and $\nu=F^*\rho$, so $S^tU>0$ iff $\nu>0$, and we are done. See \cite{bnz}.
\end{proof}

Now, let us investigate the circulant and symmetric orthogonal matrices:

\begin{lemma}
For a matrix $U\in M_N(\mathbb C)$, the following are equivalent:
\begin{enumerate}
\item $U$ is orthogonal, circulant and symmetric.

\item $U=FQF^*$ with $q\in\{\pm 1\}^N$ satisfying $q_i=q_{-i}$.
\end{enumerate}
\end{lemma}

\begin{proof}
It follows from Lemma 1.3 that $U$ is orthogonal and symmetric iff $U=FQF^*$, with $q\in\mathbb T^N$ satisfying $\bar{q}_i=q_{-i}$. The symmetry condition reads $(Fq)_i=(Fq)_{-i}$ which translates into the following system of equations, with $i=0,\ldots,N-1$:
$$\sum_kw^{ik}(q_k-q_{-k})=0$$

This system admits the unique solution $q_k-q_{-k}=0$, and the result follows.
\end{proof}

As an example, the vector $q=(-1)^n(1,-1,1,\ldots,-1,1,1,-1,\ldots,1,-1)$, having length $N=2n+1$, produces the following $N\times N$ matrix, from \cite{bnz}:
$$L_N=\frac{1}{N}
\begin{pmatrix}
1&-\cos^{-1}\frac{\pi}{N}&\cos^{-1}\frac{2\pi}{N}&\ldots&\cos^{-1}\frac{(N-1)\pi}{N}\\
\cos^{-1}\frac{(N-1)\pi}{N}&1&-\cos^{-1}\frac{\pi}{N}&\ldots&-\cos^{-1}\frac{(N-2)\pi}{N}\\
\ldots&\ldots&\ldots&\ldots&\ldots\\
-\cos^{-1}\frac{\pi}{N}&\cos^{-1}\frac{2\pi}{N}&-\cos^{-1}\frac{3\pi}{N}&\ldots&1
\end{pmatrix}$$

Let us reformulate now the above result in a more convenient form, and gather as well some examples. Recall that the integer part of a real number $r$ is denoted $\lfloor r \rfloor$.

\begin{proposition}
There are $2^n$ circulant symmetric orthogonal $N\times N$ matrices, indexed via $U=FQF^*$ by sign vectors $q\in\{\pm 1\}^n$, where $n=\lfloor N/2 \rfloor + 1$. The examples include:
\begin{enumerate}
\item The identity matrix $1_N$, coming from $q=(1,1,\ldots,1)$.

\item The matrix $U_N=2J_N-1_N$, coming from $q=(1,-1,-1,\ldots,-1)$.

\item For $N$ even, the matrix $S_N=(^0_1{\ }^1_0)$, coming from $q=(1,-1,1,\ldots,-1,1,-1)$.

\item For $N$ odd, the above matrix $L_N$, coming from $q=(-1)^{\lfloor N/2 \rfloor}(1,-1,1,\ldots,-1,1)$.
\end{enumerate}
\end{proposition}

\begin{proof}
The first assertion follows from Lemma 1.5, and from the fact that the condition $q_i=q_{-i}$ is redundant for $i=0$ for all $N$, and for $i=N/2$ when $N$ is even. The vector $q$ generating the orthogonal matrix is then given by $q=(q_0,q_1,q_2,\ldots,q_2,q_1)$.

Regarding now the assertions (1-4), we just have to prove here that the $q$-vectors in the statement produce indeed the matrices in the statement. But this is clear for (1-3), and (4) follows as well, via an elementary computation performed in \cite{bnz}.
\end{proof}

\begin{theorem}
The number of orthogonal circulant symmetric matrices (OCS), orthogonal circulant symmetric matrices with nonzero entries (OCSN) and circulant symmetric almost Hadamard matrices (AHM) is as follows:
\end{theorem}

\begin{center}
\begin{tabular}[t]{|l|l|l|l|l|l|l|l|l|l|l|l|l|l|l|l||||}
\hline $N$&2&3&4&5&6&7&8&9&10&11&12&13&14&15\\
\hline OCS&4&4&8&8&16&16&32&32&64&64&128&128&256&256\\
\hline OCSN&0&2&4&6&8&14&16&22&40&62&44&126&176&186\\
\hline AHM&0&2&4&6&8&14&8&22&24&42&36&108&104&68\\
\hline
\end{tabular}
\end{center}
\medskip

\begin{center}
\begin{tabular}[t]{|l|l|l|l|l|l|l|l|l|l|l|l|l|l|l|l||||}
\hline $N$&16&17&18&19&20&21&22&23\\
\hline OCSM&512&512&1024&1024&2048&2048&4096&4096\\
\hline OCSN&296&510&536&1022&1220&1642&3088&4094\\
\hline AHM&136&302&152&404&404&418&728&1410\\
\hline
\end{tabular}
\end{center}
\medskip

\begin{center}
\begin{tabular}[t]{|l|l|l|l|l|l|l|l|l|l|l|l|l|l|l|l||||}
\hline $N$&24&25&26&27&28&29&30\\
\hline OCS&8192&8192&16384&16384&32768&32768&65536\\
\hline OCSN&4000&7734&12688&13586&22324&32766&39080\\
\hline AHM&856&1780&2504&3098&4140&6740&5608\\
\hline
\end{tabular}
\end{center}

\medskip

\begin{proof}
This follows from a computer implementation\footnote{Source code available at \href{http://www.irsamc.ups-tlse.fr/inechita/code/ocsn-ahm.zip}{http://www.irsamc.ups-tlse.fr/inechita/code/ocsn-ahm.zip}} of the algorithm in Proposition 1.4, by using as input the vector $q\in\{\pm 1\}^n$, with $n=\lfloor N/2 \rfloor + 1$, from Proposition 1.6.
\end{proof}

Observe the arithmetic dependence of the above numbers with $N$. However, this dependence is not exact, so in order to formulate an exact conjecture here, we would probably have to take into account certain algebraic geometric multiplicities, as in Haagerup's paper \cite{ha2}. The only observation we can make at this point is that, for prime $N$, there are only two OCS matrices with zero entries, $\pm 1_N$. We intend to come back to these questions in some future work.

\section{Critical points, color decomposition}

In this section we characterize the critical points of the $p$-norm on $O(N)$. Our starting point, which motivates our study, is the following simple observation from \cite{bcs}:

\begin{proposition}
Let $U\in O(N)$, and let $p\in [1,\infty]-\{2\}$.
\begin{enumerate}
\item If $p<2$ then $||U||_p\leq N^{2/p-1/2}$, with equality iff $H=\sqrt{N}U$ is Hadamard.

\item If $p>2$ then $||U||_p\geq N^{2/p-1/2}$, with equality iff $H=\sqrt{N}U$ is Hadamard.
\end{enumerate}
\end{proposition}

\begin{proof}
In the case $p<2$, the H\"older inequality gives:
$$||U||_p\leq N^{2/p-1}||U||_2=N^{2/p-1/2}$$

Also, in the case $p>2$, the H\"older inequality gives:
$$||U||_p\geq N^{2/p-1}||U||_2=N^{2/p-1/2}$$

In both cases the equality holds when all the numbers $|U_{ij}|$ are proportional, and we conclude that we have equality if and only if $|U_{ij}|=1/\sqrt{N}$, as stated. 
\end{proof}

Observe that at $p=1,4,\infty$ we obtain $||U||_1\leq N\sqrt{N}$, $||U||_4\geq 1$, $||U||_\infty\geq 1/\sqrt{N}$, in all cases with equality if and only if the rescaled matrix $H=\sqrt{N}U$ is Hadamard.

\begin{definition}
A matrix $H\in M_N(\mathbb R)$ such that $U=H/\sqrt{N}$ is orthogonal is called:
\begin{enumerate}
\item $p$-almost Hadamard $(p<2)$, if $U$ locally maximizes the $p$-norm on $O(N)$.

\item $p$-almost Hadamard $(p>2)$, if $U$ locally minimizes the $p$-norm on $O(N)$.
\end{enumerate}
\end{definition}

As a first remark, given an exponent $p\neq 2$ and a number $N\in\{2\}\cup 4\mathbb N$ where the Hadamard Conjecture holds, the Hadamard matrices of order $N$ are precisely the $p$-almost Hadamard matrices $H\in M_N(\mathbb R)$ which are ``optimal'', in the sense that the rescaled matrix $U=H/\sqrt{N}$ is a global maximum/minimum of the $p$-norm on $O(N)$.

Let us try now to understand the critical points of the various $p$-norms on $O(N)$. Consider the set $O(N)^*\subset O(N)$ of orthogonal matrices having nonzero entries. Given a function $\varphi\in C^1(0,\infty)$, the associated function $F(U)=\sum_{ij}\varphi(|U_{ij}|)$ is differentiable around each $U\in O(N)^*$, and the critical points of $F$ can be found as follows:

\begin{lemma}
For $U\in O(N)^*$ and $\varphi\in C^1(0,\infty)$, the following are equivalent:
\begin{enumerate}
\item $U$ is a critical point of $F(U)=\sum_{ij}\varphi(|U_{ij}|)$. 

\item $WU^t$ is symmetric, where $W_{ij}={\rm sgn}(U_{ij})\varphi'(|U_{ij}|)$.
\end{enumerate}
\end{lemma}

\begin{proof}
We follow the proof in \cite{bcs}, where this result was established for $\varphi(x)=x$. We know that the group $O(N)$ consists of the zeroes of the following polynomials:
$$A_{ij}=\sum_kU_{ik}U_{jk}-\delta_{ij}$$

Also, $U$ is a critical point of $F$ iff $dF\in span(dA_{ij})$. Now since $A_{ij}=A_{ji}$, this is the same as asking for a symmetric matrix $M$ such that $dF=\sum_{ij}M_{ij}dA_{ij}$. But:
$$\sum_{ij}M_{ij}dA_{ij}
=\sum_{ijk}M_{ij}(U_{ik}dU_{jk}+U_{jk}dU_{ik})
=2\sum_{lk}(MU)_{lk}dU_{lk}$$

On the other hand, with $S_{ij}={\rm sgn}(U_{ij})$, we get:
$$dF=\sum_{ij}d\left(\varphi(S_{ij}U_{ij})\right)=\sum_{lk}S_{ij}\varphi'(S_{ij}U_{ij})dU_{ij}=\sum_{ij}W_{ij}dU_{ij}$$

We conclude that $U$ is a critical point of $F$ iff there exists a symmetric matrix $M$ such that $W=2MU$. Now by using the assumption $U\in O(N)$, this condition simply tells us that the matrix $M=WU^t/2$ must be symmetric, so we are done.
\end{proof}

In order now to investigate the symmetry property of the matrix $WU^t$ appearing in the above statement, we use the following notion:

\begin{definition}
The color decomposition of $U\in O(N)^*$ is $U=\sum_{r>o}rU^{(r)}$, where:
$$U^{(r)}_{ij}=
\begin{cases}
{\rm sgn}(U_{ij})&{\rm if}\ |U_{ij}|=r\\
0&{\rm if}\ |U_{ij}|\neq r
\end{cases}$$
The matrices $U^{(r)}\in M_N(-1,0,1)$ will be called ``color components'' of $U$. 
\end{definition}

If we let $S_{ij}={\rm sgn}(U_{ij})$, then for any $\psi:(0,\infty)\to\mathbb R$ we have the following useful formula, that we will use many times in what follows:
$$S_{ij}\psi(|U_{ij}|)=\sum_{r>0}\psi(r)U_{ij}^{(r)}$$

Let us investigate now the critical points of all $p$-norms on $O(N)$:

\begin{theorem}
For $U\in O(N)^*$, the following are equivalent:
\begin{enumerate}
\item $U$ is a critical point of the $p$-norm on $O(N)$, for any $p\in[1,\infty]$.

\item $U$ is a critical point of $F(U)=\sum_{ij}\varphi(|U_{ij}|)$, for any $\varphi\in C^1(0,\infty)$. 

\item $WU^t$ is symmetric for any $\psi:(0,\infty)\to\mathbb R$, where $W_{ij}={\rm sgn}(U_{ij})\psi(|U_{ij}|)$.

\item $U^{(r)}U^t$ is symmetric for any $r>0$, where $U^{(r)}$ are the color components of $U$.
\end{enumerate}
\end{theorem}

\begin{proof}
The result basically follows from Lemma 2.3:

(1)$\iff$(2) In one sense this is trivial, because it suffices to choose the continuously differentiable functions $\varphi(x)=x^p$. In the other sense, this follows from the fact that the functions $\varphi(x)=x^p$ span a dense subalgebra of $C^1(0,\infty)$.

(2)$\iff$(3) This follows from Lemma 2.3, because the condition found there is purely algebraic, and hence doesn't depend on the fact that $\psi=\varphi'$ is continuous.

(3)$\iff$(4) We have the following formula:
$$(WU^t)_{ij}=\sum_k{\rm sgn}(U_{ik})\psi(|U_{ik}|)U_{jk}=\sum_{r>0}\psi(r)\sum_{k,|U_{ik}|=r}{\rm sgn}(U_{ik})U_{jk}$$

In terms of the color components of $U$, this formula becomes:
$$(WU^t)_{ij}=\sum_{r>0}\psi(r)\sum_kU^{(r)}_{ik}U_{jk}=\sum_{r>0}\psi(r)(U^{(r)}U^t)_{ij}$$

Thus the matrix appearing in (2) is simply given by:
$$WU^t=\sum_{r>0}\psi(r)U^{(r)}U^t$$

Now since $\psi:(0,\infty)\to\mathbb R$ can be here any function, the result follows.
\end{proof}

As a first consequence, we have:

\begin{corollary}
Let $U=U(x,y)$ be orthogonal, coming from an $(a,b,c)$ pattern. Then $U$ is a critical point of all the $p$-norms on $O(N)$.
\end{corollary}

\begin{proof}
As explained in \cite{bnz} the fact that $U$ is orthogonal shows that $x,y$ have opposite signs, and we will make the same normalization as there, namely $x<0$, $y>0$. 

Consider now the matrices $U_x,U_y\in M_N(0,1)$ describing the positions where our variables $x,y$ sit inside $U$. Then we have the following formulae:
$$U^{(x)}=-U_x,\quad U^{(y)}=U_y,\quad U=xU_x+yU_y,\quad U_x+U_y=NJ_N$$

By using these formulae we obtain that $U^{(x)}U^t$ is indeed self-adjoint:
\begin{eqnarray*}
U^{(x)}U^t
&=&-U_x(xU_x^t+yU_y^t)\\
&=&-U_x(xU_x^t+y(NJ_N-U_x^t))\\
&=&(y-x)U_xU_x^t-yNJ_NU_x\\
&=&(y-x)U_xU_x^t-y(a+b)NJ_N
\end{eqnarray*}

A similar computation shows that $U^{(y)}U^t$ is self-adjoint as well, and we are done.
\end{proof}

We have as well the following consequence:

\begin{corollary}
Any circulant and symmetric matrix $U\in O(N)$ having nonzero entries is a critical point of all $p$-norms on $O(N)$.
\end{corollary}

\begin{proof}
For a color $r>0$, consider the set of indices where this color appears on the first row, $D_r=\{k||\gamma_k|=r\}$. From $\gamma_i=\gamma_{-i}$ we get $D_r=-D_r$, and so:
$$(U^{(r)}U^t)_{ij}
=\sum_kU^{(r)}_{ik}U_{jk}
=\sum_{s\in D_r}\mathrm{sgn}(\gamma_s)\gamma_{s+i-j}
=\sum_{t\in D_r}\mathrm{sgn}(\gamma_t)\gamma_{t+j-i}
=(U^{(r)}U^t)_{ji}$$

This shows that $U^{(r)}U^t$ is symmetric, and we are done. 
\end{proof}

The above corollary can be regarded as a slight advance on a key problem raised in \cite{bnz}, namely that of characterizing the circulant almost Hadamard matrices.

We have as well the following question, that we believe of interest:

\begin{problem}
What are the matrices $U\in O(N)^*$ having the property that $U^{(r)}(U^{(s)})^t$ is symmetric for any $r,s$, where $U=\sum_{r>0}rU^{(r)}$ is the color decomposition?
\end{problem}

The point is that all the examples of joint critical points of all $p$-norms on $O(N)$ that we have, namely the rescaled Hadamard matrices, the matrices coming from $(a,b,c)$ patterns, and the circulant and symmetric matrices, satisfy in fact this stronger condition.

Observe also that the condition in Problem 2.8, involving just $-1,0,1$ matrices, is purely combinatorial. In fact, what we have there is an axiomatization of some new ``design-type'' combinatorial structure, generalizing the Hadamard matrices.

\section{Local extrema, the rotation trick}

In this section and in the next one we find an algebraic criterion for detecting the $p$-almost Hadamard matrices, by building on the previous work in \cite{bcs} at $p=1$.

The result will basically come from the computation of the Hessian of the $p$-norm on $O(N)$. However, since this $p$-norm is in general not differentiable at points $U\in O(N)$ having zero entries, we first must prove that the local extrema belong to $O(N)^*$.

At $p=1$ this was done in \cite{bcs}, by using a ``rotation trick''. The same trick works in fact at any $p<2$, but with some more calculus needed afterwards, and we have:

\begin{theorem}
If $U\in O(N)$ is a local maximum of the $p$-norm on $O(N)$, for some exponent $p\in[1,2)$, then $U\in O(N)^*$.
\end{theorem}

\begin{proof}
Let $U_1,\ldots,U_N$ be the columns of $U$, and let us perform a rotation of $U_1,U_2$:
$$\begin{pmatrix}U^t_1\\ U^t_2\end{pmatrix}=\begin{pmatrix}
\cos t\cdot U_1-\sin t\cdot U_2\\ \sin t\cdot U_1+\cos t\cdot U_2
\end{pmatrix}$$

In order to compute the $p$-norm, let us permute the columns of $U$, in such a way that the first two rows look as follows, with $X_k\neq 0$, $Y_k\neq 0$, $A_kC_k>0$, $B_kD_k<0$:
$$\begin{pmatrix}U_1\\ U_2\end{pmatrix}
=\begin{pmatrix}
0&0&Y&A&B\\
0&X&0&C&D
\end{pmatrix}$$

Let us compute now the following quantity:
\begin{eqnarray*}
\varphi(t)
&=&||U^t||_p^p-||U||_p^p\\
&=&||\cos t\cdot U_1-\sin t\cdot U_2||_p^p+||\sin t\cdot U_1+\cos t\cdot U_2||_p^p-||U_1||_p^p-||U_2||_p^p
\end{eqnarray*}

We have the folowing formula:
\begin{eqnarray*}
\varphi(t)
&=&||\sin t\cdot X||_p^p+||\cos t\cdot Y||_p^p+||\cos t\cdot A-\sin t\cdot C||_p^p+||\cos t\cdot B-\sin t\cdot D||_p^p\\
&+&||\cos t\cdot X||_p^p+||\sin t\cdot Y||_p^p+||\sin t\cdot A+\cos t\cdot C||_p^p+||\sin t\cdot B+\cos t\cdot D||_p^p\\
&-&||X||_p^p-||Y||_p^p-||A||_p^p-||B||_p^p-||C||_p^p-||D||_p^p
\end{eqnarray*}

Thus for $t>0$ small we have:
\begin{eqnarray*}
\varphi(t)
&=&(\sin^pt+\cos^pt-1)(||X||_p^p+||Y||_p^p)\\
&+&||\cos t\cdot A-\sin t\cdot C||_p^p+||\sin t\cdot A+\cos t\cdot C||_p^p-||A||_p^p-||C||_p^p\\
&+&||\cos t\cdot B-\sin t\cdot D||_p^p+||\sin t\cdot B+\cos t\cdot D||_p^p-||B||_p^p-||D||_p^p
\end{eqnarray*}

Now by remembering our conventions $A_kC_k>0$, $B_kD_k<0$, we obtain:
\begin{eqnarray*}
\varphi(t)
&=&(\sin^pt+\cos^pt-1)(||X||_p^p+||Y||_p^p)\\
&+&\sum_k(\cos t|A_k|-\sin t|C_k|)^p+(\cos t|C_k|+\sin t|A_k|)^p-|A_k|^p-|C_k|^p\\
&+&\sum_k(\cos t|B_k|+\sin t|D_k|)^p+(\cos t|D_k|-\sin t|B_k|)^p-|B_k|^p-|D_k|^p
\end{eqnarray*}

Consider now the matrix $V$ obtained by interchanging $U_1,U_2$.  If we perform to it a rotation as above, then the quantity $\psi(t)=||V^t||_p^p-||V||_p^p$ is given by:
\begin{eqnarray*}
\psi(t)
&=&(\sin^pt+\cos^pt-1)(||X||_p^p+||Y||_p^p)\\
&+&\sum_k(\cos t|C_k|-\sin t|A_k|)^p+(\cos t|A_k|+\sin t|C_k|)^p-|A_k|^p-|C_k|^p\\
&+&\sum_k(\cos t|D_k|+\sin t|B_k|)^p+(\cos t|B_k|-\sin t|D_k|)^p-|B_k|^p-|D_k|^p
\end{eqnarray*}

Let us introduce now the following function $\gamma_t$, depending on $a,c\geq 0$:
\begin{eqnarray*}
\gamma_t(a,c)
&=&(\cos t\cdot a+\sin t\cdot c)^p+(\cos t\cdot c+\sin t\cdot a)^p\\
&+&(\cos t\cdot a-\sin t\cdot c)^p+(\cos t\cdot c-\sin t\cdot a)^p\\
&-&2a^p-2c^p
\end{eqnarray*}

With this notation, if we sum the above two formulae of $\varphi,\psi$, we obtain:
\begin{eqnarray*}
\varphi(t)+\psi(t)
&=&2(\sin^pt+\cos^pt-1)(||X||_p^p+||Y||_p^p)\\
&+&\sum_k\gamma_t(|A_k|,|C_k|)+\sum_k\gamma_k(|B_k|,|D_k|)
\end{eqnarray*}

Now observe that the derivative of this quantity is given by:
\begin{eqnarray*}
\varphi'(t)+\psi'(t)
&=&2p(\sin^{p-1}t\cos t-\cos^{p-1}t\sin t)(||X||_p^p+||Y||_p^p)\\
&+&\sum_k\gamma_t'(|A_k|,|C_k|)+\sum_k\gamma_k'(|B_k|,|D_k|)
\end{eqnarray*}

So, let us compute now the derivative of $\gamma_t$:
\begin{eqnarray*}
\gamma_t'(a,c)
&=&p(\cos t\cdot a+\sin t\cdot c)^{p-1}(-\sin t\cdot a+\cos t\cdot c)\\
&+&(\cos t\cdot c+\sin t\cdot a)^{p-1}(-\sin t\cdot c+\cos t\cdot a)\\
&+&(\cos t\cdot a-\sin t\cdot c)^p(-\sin t\cdot a-\cos t\cdot c)\\
&+&(\cos t\cdot c-\sin t\cdot a)^p(-\sin t\cdot c-\cos t\cdot a)
\end{eqnarray*}

By using $\sin t=t+O(t^2)$ and $\cos t=1+O(t^2)$ we obtain:
\begin{eqnarray*}
\gamma_t'(a,c)
&\simeq&p(a+tc)^{p-1}(c-ta)+p(c+ta)^{p-1}(a-tc)\\
&-&p(a-tc)^{p-1}(c+ta)-p(c-ta)^{p-1}(a+tc)
\end{eqnarray*}

By using the power series expansion for the exponentials, this gives:
\begin{eqnarray*}
\frac{\gamma_t'(a,c)}{p}
&\simeq&(a^{p-1}+(p-1)a^{p-2}tc)(c-ta)+(c^{p-1}+(p-1)c^{p-2}ta)(a-tc)\\
&-&(a^{p-1}-(p-1)a^{p-2}tc)(c+ta)-(c^{p-1}-(p-1)c^{p-2}ta)(a+tc)
\end{eqnarray*}

The order 0 terms cancel, and by neglecting the order 2 terms we obtain:
\begin{eqnarray*}
\frac{\gamma_t'(a,c)}{p}
&\simeq&((p-1)a^{p-2}c^2-a^p)t+((p-1)c^{p-2}a^2-c^p)t\\
&-&(a^p-(p-1)a^{p-2}c^2)t-(c^p-(p-1)c^{p-2}a^2)t
\end{eqnarray*}

Now since the upper and lower terms are the same, we obtain:
\begin{eqnarray*}
\frac{\gamma_t'(a,c)}{2pt}
&\simeq&(p-1)a^{p-2}c^2-a^p+(p-1)c^{p-2}a^2-c^p\\
&=&(p-1)(a^{p-2}c^2+a^2c^{p-2})-(a^p+c^p)
\end{eqnarray*}

With these formulae in hand, we claim that $X,Y$ both follow to be null vectors. Indeed, since we are in the case $p\in[1,2)$, the matrices $U,V$ are local maximizers of the $p$-norm. Thus $\varphi,\psi\leq 0$ for $t>0$ small, so we must have $\varphi'+\psi'\leq 0$ for $t>0$ small. But:
$$\varphi'(t)+\psi'(t)=2pt^{p-1}(||X||_p^p+||Y||_p^p)+O(t)$$

Thus we have $||X||_p^p+||Y||_p^p\leq 0$, and so $X,Y$ are both null vectors, as claimed.

Summarizing, we have proved that the $0$ entries of $U_1,U_2$ must appear at the same positions. By permuting the rows of $U$ the same must hold for any two rows $U_i,U_j$. Now since $U\in O(N)$ cannot have zero columns, all its entries must be nonzero, as claimed.
\end{proof}

It is not clear whether the same holds at $p\in(2,\infty)$. Here $U,V$ are local minimizers of the $p$-norm, so $\varphi,\psi\geq 0$ for $t>0$ small, so $\varphi'+\psi'\geq 0$ for $t>0$ small. But:
$$\varphi'(t)+\psi'(t)=-2pt(||X||_1+||Y||_1)+2ptS_{ABCD}+O(t^{1+\varepsilon})$$

Here $S_{ABCD}$ is a sum of quantities of the following type, one for each pair of adjacent entries of $A,C$, and one for each pair of adjacent entries of $B,D$:
$$K(a,c)=(p-1)(a^{p-2}c^2+a^2c^{p-2})-(a^p+c^p)$$

The problem comes from the fact that these quantities, and hence their sum $S_{ABCD}$ as well, can be positive, so that we cannot conclude that we have $||X||_p^p+||Y||_p^p\leq 0$.

The case $p=\infty$ is also very problematic, because when the maximum $M=\max|U_{ij}|$ appears at many places in our matrix, the rotation trick obviously cannot work. In fact, there are many problems here, and the rotation trick at $p=\infty$ seems to require precise information about the positions of the $M$ and $0$ entries in our matrix.

Of course, the fact that the rotation trick might fail at $p\in(2,\infty]$ is not an indication that the conclusion $U\in O(N)^*$ should fail itself, but just of the fact that the good rotation $U^t=Ue^{tA}$ might come from more complicated antisymmetric matrices $A\in M_N(\mathbb R)$.

Here is an example of such a result, excluding a few matrices having zero entries:

\begin{proposition}
An antisymmetric matrix $A \in O(N)$ cannot be a local extremum of the $p$-norm on $O(N)$, for any $p\geq 1$.
\end{proposition}

\begin{proof}
Since $A$ is orthogonal and antisymmetric, we have $A^2=-AA^t=-1$, and so:
$$e^{tA}=\cos t\cdot A-\sin t\cdot 1_N$$

We analyze, to the first order in $t\to 0$, the following function:
$$||Ae^{tA}||_p^p-||A||_p^p = (|\cos t|^p-1)||A||_p^p+N|\sin t|^p$$

At $p<2$ this function behaves like $N|t|^p$, so $A$ cannot be a local maximum for the $p$-norm, since it is a local minimum in the direction $A$. Similarly, at $p>2$ the norm difference behaves like $-p|t|^2/2$, so $A$ cannot be a local minimum for the $p$-norm.
\end{proof}

\section{The Hessian formula, open problems}

In this section we find an algebraic criterion for detecting the $p$-almost Hadamard matrices. For this purpose, let us first go back to Theorem 2.5 above, and introduce:

\begin{definition}
To any $U\in O(N)^*$ we associate the matrices $L_r=U^{(r)}U^t$ and $R_r=U^tU^{(r)}$, where $U=\sum_{r>0}rU^{(r)}$ is the color decomposition of $U$.
\end{definition}

According to Theorem 2.5 above, in the case where $U$ is a critical point of all the $p$-norms on $O(N)$, the matrices $L_r$ are all symmetric, and the matrices $R_r=U^tL_rU$ follow to be symmetric too. Observe also that we have the following formula:
$$\sum_{r>0}rL_r=\sum_{r>0}rR_r=1$$

We now study the local extrema of the $p$-norm on $O(N)^*$. We use:

\begin{lemma}
Let $U\in O(N)^*$, let $p\in[1,\infty)$, and for $A\in M_N(\mathbb R)$ antisymmetric, set:
$$\varphi(A)=\sum_{ij}|U_{ij}|^{p-2}((p-1)(UA)_{ij}^2+U_{ij}(UA^2)_{ij})$$
Then $U$ is a local maximizer/minimizer of the $p$-norm iff $\sum_{r>0}r^{p-1}Tr(R_rA^t)=0$, and the quantity $\varphi(A)$ is positive/negative, for any $A\in M_N(\mathbb R)$ antisymmetric.
\end{lemma}

\begin{proof}
Since the Lie algebra of $SO(N)$ consists of the antisymmetric matrices $A\in M_N(\mathbb R)$, in the neighborhood of $U\in O(N)$ we have matrices of type $Ue^{tA}$, with $A$ antisymmetric, and with $t\in\mathbb R$ close to $0$. So, let us fix $A\in M_N(\mathbb R)$ antisymmetric, and set:
$$f(t)=||Ue^{tA}||_p^p$$

With $S_{ij}={\rm sgn}(U_{ij})$, for $t\in\mathbb R$ close enough to $0$ we have:
$$f(t)=\sum_{ij}|(Ue^{tA})_{ij}|^p=\sum_{ij}(S_{ij}(Ue^{tA})_{ij})^p$$

Now the derivative of this function with respect to $t$ is given by:
\begin{eqnarray*}
f'(t)
&=&\sum_{ij}p|(Ue^{tA})_{ij}|^{p-1}S_{ij}(Ue^{tA})_{ij}'\\
&=&\sum_{ijk}p|(Ue^{tA})_{ij}|^{p-1}S_{ij}U_{ik}(e^{tA})_{kj}'\\
&=&\sum_{ijk}pS_{ij}U_{ik}|(Ue^{tA})_{ij}|^{p-1}(e^{tA})_{kj}'
\end{eqnarray*}

In particular at $t=0$ we obtain the following quantity, whose vanishing corresponds to the first condition in the statement:
\begin{eqnarray*}
f'(0)
&=&p\sum_{ijk}S_{ij}U_{ik}|U_{ij}|^{p-1}A_{kj}\\
&=&p\sum_{ij}S_{ij}|U_{ij}|^{p-1}(UA)_{ij}\\
&=&p\sum_{r>0}c^{r-1}Tr(R_rA^t)
\end{eqnarray*}

Also by using the above formula, let us compute now the second derivative:
$$f''(t)=\sum_{ijk}pS_{ij}U_{ik}((p-1)|(Ue^{tA})_{ij}|^{p-2}S_{ij}(Ue^{tA})_{ij}'(e^{tA})_{kj}'+|(Ue^{tA})_{ij}|^{p-1}(e^{tA})_{kj}'')$$

At $t=0$ now, by using $(e^{tB})'_{|t=0}=B$ for any $B\in M_N(\mathbb R)$, we get:
\begin{eqnarray*}
f''(0)
&=&\sum_{ijk}pS_{ij}U_{ik}((p-1)|U_{ij}|^{p-2}S_{ij}(UA)_{ij}A_{kj}+|U_{ij}|^{p-1}(A^2)_{kj})\\
&=&\sum_{ijk}pU_{ik}|U_{ij}|^{p-2}((p-1)(UA)_{ij}A_{kj}+U_{ij}(A^2)_{kj})\\
&=&\sum_{ij}p|U_{ij}|^{p-2}((p-1)(UA)_{ij}(UA)_{ij}+U_{ij}(UA^2)_{ij})
\end{eqnarray*}

Thus we have $f''(0)=p\varphi(A)$, and this gives the result.
\end{proof}

\begin{theorem}
A matrix $U\in O(N)^*$ is a local maximizer/minimizer of the $p$-norm on $O(N)$, with $p\in[1,\infty)$, if and only if the matrix $\sum_{r>0}r^{p-1}L_r$ is symmetric, and with
\begin{eqnarray*}
Y_{ab,cd}
&=&\delta_{bd}\sum_i((p-1)|U_{ib}|^{p-2}-|U_{ic}|^{p-2})U_{ia}U_{ic}\\
&-&\delta_{bc}\sum_i((p-1)|U_{ib}|^{p-2}-|U_{id}|^{p-2})U_{ia}U_{id}\\
&-&\delta_{ad}\sum_i((p-1)|U_{ia}|^{p-2}-|U_{ic}|^{p-2})U_{ib}U_{ic}\\
&+&\delta_{ac}\sum_i((p-1)|U_{ia}|^{p-2}-|U_{id}|^{p-2})U_{ib}U_{id}
\end{eqnarray*}
the quadratic form $\varphi=\sum_{abcd}Y_{ab,cd}B_{ab}B_{cd}$ is positive/negative.
\end{theorem}

\begin{proof}
Let us look at the two conditions found in Lemma 4.2. The first condition, namely that we have $\sum_{r>0}r^{p-1}Tr(R_rA^t)=0$ for any $A\in M_N(\mathbb R)$ antisymmetric, is equivalent to the first condition in the statement, namely that the matrix $\sum_{r>0}r^{p-1}L_r$ is symmetric.

The quantity found in Lemma 4.2 can be written as:
\begin{eqnarray*}
\varphi
&=&\sum_{ij}|U_{ij}|^{p-2}((p-1)(UA)_{ij}^2+U_{ij}(UA^2)_{ij})\\
&=&\sum_{ijkl}|U_{ij}|^{p-2}((p-1)U_{ik}U_{il}A_{kj}A_{lj}+U_{ij}U_{ik}A_{kl}A_{lj})\\
&=&\sum_{ijkl}|U_{ij}|^{p-2}(p-1)U_{ik}U_{il}A_{kj}A_{lj}-\sum_{ijkl}|U_{il}|^{p-2}U_{il}U_{ik}A_{kj}A_{lj}\\
&=&\sum_{ijkl}((p-1)|U_{ij}|^{p-2}-|U_{il}|^{p-2})U_{ik}U_{il}\cdot A_{kj}A_{lj}
\end{eqnarray*}

Now recall that $A\in M_N(\mathbb R)$ is an arbitrary antisymmetric matrix. So, let us write $A=B-B^t$. In terms of the matrix $B\in M_N(\mathbb R)$, which can be arbitrary, we have:
$$\varphi=\sum_{ijkl}((p-1)|U_{ij}|^{p-2}-|U_{il}|^{p-2})U_{ik}U_{il}(B_{kj}-B_{jk})(B_{lj}-B_{jl})$$

The expression on the right in the above formula is:
\begin{eqnarray*}
X
&=&(B_{kj}-B_{jk})(B_{lj}-B_{jl})\\
&=&\sum_{abcd}(\delta_{kj,ab}-\delta_{jk,ab})(\delta_{lj,cd}-\delta_{jl,cd})B_{ab}B_{cd}\\
&=&\sum_{abcd}(\delta_{bd}\delta_{jkl,bac}-\delta_{bc}\delta_{jkl,bad}-\delta_{ad}\delta_{jkl,abc}+\delta_{ac}\delta_{jkl,abd})B_{ab}B_{cd}
\end{eqnarray*}

It follows that our map $\varphi$ is given by:
\begin{eqnarray*}
\varphi
&=&\sum_{abcd}B_{ab}B_{cd}\sum_{ijkl}((p-1)|U_{ij}|^{p-2}-|U_{il}|^{p-2})U_{ik}U_{il}\\
&&(\delta_{bd}\delta_{jkl,bac}-\delta_{bc}\delta_{jkl,bad}-\delta_{ad}\delta_{jkl,abc}+\delta_{ac}\delta_{jkl,abd})
\end{eqnarray*}

But this gives the formula in the statement, and we are done.
\end{proof}

Observe that at $p=2$ we have $\varphi=0$. Also, in the case where the rescaled matrix $H=\sqrt{N}U$ is Hadamard, the rescaled matrix $\widetilde{Y}=(\sqrt{N})^{p-2}Y$ is given by:
\begin{eqnarray*}
\widetilde{Y}_{ab,cd}
&=&(p-1)\delta_{bd}\sum_iU_{ia}U_{ic}-(p-1)\delta_{bc}\sum_iU_{ia}U_{id}-\delta_{ad}\sum_iU_{ib}U_{ic}+\delta_{ac}\sum_iU_{ib}U_{id}\\
&=&(p-1)\delta_{bd}\delta_{ac}-(p-1)\delta_{bc}\delta_{ad}-\delta_{ad}\delta_{bc}-\delta_{ac}\delta_{bd}\\
&=&(p-2)(\delta_{ac}\delta_{bd}-\delta_{ad}\delta_{bc})
\end{eqnarray*}

Thus the rescaled quadratic form $\widetilde{\varphi}=(\sqrt{N})^{p-2}\varphi$ is given by:
\begin{eqnarray*}
\widetilde{\varphi}_{ab,cd}
&=&(p-2)\sum_{abcd}(\delta_{ac}\delta_{bd}-\delta_{ad}\delta_{bc})B_{ab}B_{cd}\\
&=&(p-2)\sum_{ab}(B_{ab}^2-B_{ab}B_{ba})\\
&=&(p-2)\cdot\frac{1}{2}\sum_{ab}(B_{ab}-B_{ba})^2
\end{eqnarray*}

These computations agree of course with the fact that the $2$-norm is constant on $O_N$, and that the multiples of Hadamard matrices are $p$-almost Hadamard, for any $p$.

Observe that some simplifications appear as well at $p=1$. Here we obtain of course the fact that the matrix $SU^t$ must be positive, as stated in Definition 1.1 above.

In general, the formula in Theorem 4.3 is quite a theoretical one, but can be used on a computer. As an example of potential application, our computer simulations suggest:

\begin{conjecture}
$K_N$ is $p$-almost Hadamard, for any $N$ and $p$.
\end{conjecture}

Regarding a possible direct proof, let $U_N=K_N/\sqrt{N}$, and observe first that for $U\in O(N)$ we have $NJ_NU=(S_j)_{ij}$, where $S_i$ are the sums on the columns of $U$, so:
$$(U_NU)_{ij}=\frac{2}{N}S_j-U_{ij}$$

Thus, we have the following formula for the $p$-norm of a perturbation of $U_N$:
$$||U_NU||_p^p=\sum_{ij}\Big|U_{ii}-\frac{2}{N}S_i\Big|^p$$

The problem is to prove that this quantity is locally minimized/maximized at $U=1_N$. This looks like a quite tricky problem, and we don't have results.

We have as well a series of questions concerning some possible extensions of this conjecture. We know from Corollary 2.6 and from Corollary 2.7 that both classes of matrices ``coming from designs'' and ``circulant and symmetric'' are critical points of all $p$-norms. We believe that the good framework is the ``circulant design'' one, and we have:

\begin{problem}
Consider the matrices in $O(N)$ coming from circulant designs.
\begin{enumerate}
\item What are these matrices, combinatorially speaking?

\item Which of these matrices have nonzero entries?

\item When are these matrices $p$-almost Hadamard?
\end{enumerate}
\end{problem}

In relation with question (1), one remark is that the Fano plane matrix is indeed circulant, so the answer to the problem is certainly not trivial. Question (2) looks easy but is probably not entirely trivial, because we have to exclude here for instance the identity matrix $1_N$. As for (3), this is definitely not trivial, among others because an answer here would probably require a serious combinatorial input, coming from (1).

\end{document}